\documentclass[12pt,twoside,reqno]{amsart}

\newtheorem{thm}{Theorem}[section]
\newtheorem{lemma}{Lemma}[section]
\newtheorem{rem}{Remark}[section]
\newtheorem{prop}{Proposition}[section]

\def \R{\mathbb{R}}

\numberwithin{equation}{section}

\begin{document}

\title[Global regular solutions for the 3D  ZK equation]
{
Global regular solutions for the 3D Zakharov-Kuznetsov equation posed on a bounded domain
}
\author[
 N.~A. Larkin]
{
 N.~A. Larkin}

\bigskip
\address
{
Departamento de Matem\'atica\\
Universidade Estadual de Maring\'a\\
87020-900, Maring\'a - PR, Brazil.
}
\email{ \ \  nlarkine@uem.br}
\date{}


\begin{abstract}
An initial-boundary value problem for the 3D Zakharov-Kuznetsov equation posed on bounded domains is considered.
Existence and uniqueness of a global regular solution as well as  exponential decay   of the $H^2$-norm for small initial data are proven.
\end{abstract}

\maketitle

\section{Introduction}\label{introduction}

We are concerned with the existence, uniqueness and exponential decay of the $H^2$-norm for global regular solutions   to an initial-boundary value problem (IBVP) for the 3D Zakharov-Kuznetsov (ZK) equation 
 \begin{equation}
u_t+(c_s +u)u_x +u_{xxx}+u_{xyy}+u_{xzz}=0 \label{zk}
\end{equation}
which describes the propagation of nonlinear ionic-sonic waves in a plasma submitted to a magnetic field directed along the $x$ axis and  $c_s$ is a positive constant corresponding to the sound velocity \cite{temam,temam2,zk} . This equation is a three-dimensional
analog of the well-known Korteweg-de Vries (KdV) equation
\begin{equation}\label{kdv}
u_t+uu_x+u_{xxx}=0.
\end{equation}

Equations \eqref{zk} and \eqref{kdv} are typical examples
of so-called
dispersive equations which attract considerable attention
of both pure and applied mathematicians in the past decades. The KdV
equation is probably most studied in this context.
The theory of the initial-value problem
(IVP henceforth)
for \eqref{kdv} is considerably advanced today
\cite{bona2,bourgain2,kato,ponce2,saut2,temam1}.

Recently, due to physics and numerics needs, publications on initial-boundary value
problems both  in bounded and unbounded domains for dispersive equations have appeared
\cite{bona1,bona3,bubnov,colin,larkin,lar2,rivas,wang,zhang}. In
particular, it has been discovered that the KdV equation posed on a
bounded interval possesses an implicit internal dissipation. This allowed
to prove the exponential decay rate of small solutions for
\eqref{kdv} posed on bounded intervals without adding any
artificial damping term \cite{larkin}. Similar results were proved
for a wide class of dispersive equations of any odd order with one
space variable \cite{familark}.

However, \eqref{kdv} is a satisfactory approximation for real waves phenomena while the
equation is posed on the whole line ($x\in\mathbb{R}$); if
cutting-off domains are taken into account, \eqref{kdv} is no longer
expected to mirror an accurate rendition of reality. The correct
equation in this case (see, for instance, \cite{bona1,zhang})
should be written as
\begin{equation}\label{1.3}
u_t+ u_x+uu_x+u_{xxx}=0.
\end{equation}
Indeed, if $x\in\R,\ t>0$,
the linear traveling term $u_x$ in \eqref{1.3} can be easily scaled
out by a simple change of variables, but it can
not be safely ignored for problems posed  both on finite and semi-infinite intervals without
changes in the original domain.

Once bounded domains are considered
as a spatial region of waves propagation, their sizes appear to be
restricted by certain critical conditions.
 We recall, however, that
if the transport term $u_x$ is neglected, then \eqref{1.3} becomes \eqref{kdv}, and it is possible to prove the
exponential decay rate of small solutions for \eqref{kdv} posed
on any bounded interval.
More  results on control and stabilizability for the KdV equation can
be found in \cite{rosier1,rozan}.

Later, the interest on dispersive equations became to be extended for the
multi-dimensional models such as Kadomtsev-Petviashvili (KP)
and ZK equations.
As far as the ZK equation is concerned,
 results  both on IVP and IBVP can be found in
\cite{faminski,faminski2,farah,pastor,pastor2,saut,ribaud,temam}.
The biggest part of these publications is devoted to study of well-posedness of the Cauchy problem and initial-boundary value problems for the 2D ZK equation
\cite{faminski, faminski2, farah,pastor, pastor2}. In the case of the 3D ZK equation, there are results on local well- posedness for the Cauchy problem \cite{saut,ribaud}; the existence of local strong solutions to an initial- boundary value problem posed on a bounded domain, \cite{wang}, as well as the existence of global weak solutions \cite{temam}.

 Our work has
been inspired by \cite{temam,wang} where \eqref{zk} posed  on a bounded domain was considered. A thorough analysis of these papers has revealed that an implicit dissipativity of the terms $u_{xyy}+u_{xzz}$ may help to establish a global well-posedness of initial-boundary value problems in classes of regular solutions. Yearlier this dissipativity has been used in order to prove exponential decay for the 2D ZK equation \cite{larh1,larkintronco}.

The main goal of our work is to prove the existence and uniqueness
of global-in-time regular solutions of \eqref{zk} posed  on
bounded domains and the exponential decay rate of
these solutions for sufficiently small initial data. To cope with this problem, we exploited the strategy completely different from the standard schemes: first to prove the existence result and after that to study uniqueness and decay properties of solutions. In our case, we prove simultaneously existence of global regular solutions and their exponential decay.

The paper is outlined as follows. Section I is Introduction. Section 2 contains  formulation
of the problem and auxiliaries. In Section \ref{existence}, we prove the existence and uniqueness of global regular solutions and, simultaneously, exponential decay of the $H^2$-norm  establishing  global estimates of local strong solutions provided by \cite{wang}.

\section{Problem and preliminaries}\label{problem}

Let $L,B_y,B_z$ be finite positive numbers. Define
\begin{align*}
&\mathcal{D}=\{(x,y,z)\in\mathbb{R}^3: \ x\in(0,L),\ y\in(0,B_y), \ z\in(0,B_z)\};\\
&\mathcal{S}=\{(y,z)\in \mathbb{R}^2: \ y\in (0,B_y),\ z\in(0,B_z)\},\ \ \ \mathcal{Q}_t=\mathcal{D}\times (0,t).
\end{align*}

Consider the following IBVP:
\begin{align}
&Au\equiv u_t+(c_s+u)u_x+\Delta u_x=0 \quad \mbox{in}\quad \mathcal{Q}_t;  \label{2.1}\\
&u|_{\gamma}=0,\ t>0; \label{2.2}\\
&u_x(L,y,z,t)=0,\; y\in(0,B_y),\;z\in(0,B_z),\; t>0;
\label{2.3}
\\
&u(x,y,z,0)=u_0(x,y,z),\ \ (x,y,z)\in\mathcal{D},
\label{2.4}
\end{align}
where $\gamma$ denotes the boundary of $\mathcal{D}$,\quad$u_0:\mathcal{D}\to\mathbb{R}$ is a given function.

Hereafter subscripts $u_x,\ u_{xy},$ etc. denote the partial derivatives,
as well as $\partial_x$ or $\partial_{xy}^2$ when it is convenient.
Operators $\nabla$ and $\Delta$ are the gradient and Laplacian acting over $\mathcal{D}.$
By $(\cdot,\cdot)$ and $\|\cdot\|$ we denote the inner product and the norm in $L^2(\mathcal{D}),$
and $\|\cdot\|_{H^k}$ stands for the norm in $L^2$-based Sobolev spaces.

We will need the following result \cite{lady}.
\begin{lemma}\label{lemma1}
Let $u\in H^1(\mathcal{D})$ and $\gamma$ be the boundary of $\mathcal{D}.$

If $u|_{\gamma}=0,$ then
\begin{equation}\label{lady1}
\|u\|_{L^q(\mathcal{D})}\le 4^{\theta}\|\nabla u\|^{\theta}\|u\|^{1-\theta},
\end{equation}
where  $\theta=3\left(\frac12-\frac1q\right).$

If $u|_{\gamma}\ne0,$ then
\begin{equation}\label{lady2}
\|u\|_{L^q(\mathcal{D})}\le4^{\theta} C_{\mathcal{D}}\|u\|^{\theta}_{H^1(\mathcal{D})}\|u\|^{1-\theta},
\end{equation}
where $C_{\mathcal{D}}$ does not depend on a size of $\mathcal{D}.$
\end{lemma}

\section{Existence theorem}\label{existence}

In this section we state the existence result for  bounded domains.
\begin{thm}\label{theorem1}
Let  $u_0$ be a given function such that $u_0|_{\gamma}=u_{0x}|_{x=L}=0$ and

$$\|u_0\|^2+\|u_{0yy}\|^2+\|u_{0zz}\|^2+J_0<\infty,$$
where

$$ J_0=\left((1+x),u^2_0+[(c_s+u_0)u_{0x}+\Delta  u_{0x}]^2\right).$$
Moreover,  the following conditions to be fulfilled:
\begin{align}
&K_2=\pi^2\big[\frac{7}{8B_y^2}+\frac{7}{8B_z^2}+\frac{23}{8L^2}\big]\geq 4c_s;\notag\\
& \|u_0\|^4\leq\frac{K_2}{4K_3};\quad
 J^2_0\leq \frac{K_2}{4K_4}, \label {assump}
\end{align}
where
$$K_3=3^3 2^{16}(1+L)^4(2C_1^2+1),\: C_1=1+c_s+\frac{2^{11}}{3}\|u_0\|^4,\:K_4=\frac{3^3 2^{19}}{25}(1+L)^6.$$
Then there exists a unique regular solution to
\eqref{2.1}-\eqref{2.4} such that
\begin{align*}
&u\in L^{\infty}(0,\infty;H^2(\mathcal{D}))\cap L^2(0,\infty;H^3(\mathcal{D}));\\
&\Delta u_x\in L^{\infty}(0,\infty;L^2(\mathcal{D}))\cap L^2(0,\infty;H^1(\mathcal{D}));\\
&u_t\in L^{\infty}(0,\infty;L^2(\mathcal{D}))\cap L^2(0,\infty;H^1(\mathcal{D}))
\end{align*}
and
\begin{equation}
\|u\|^2_{H^2(\mathcal{D})}(t)\leq Ce^{-\chi t}J_0.
\end{equation}
where the constant $C$ depends on $L, J_0;$\; $\chi=\frac{K_2}{4(1+L)}.$
\end{thm}
\begin{proof}

To prove this theorem,  we use  local in $t$ existence of strong solutions to \eqref{2.1}-\eqref{2.4} established in \cite{wang} and  prove global a priori estimates of strong solutions. Of course, it is possible to use a parabolic regularization as in \cite{larkintronco,wang} and to prove directly global estimates of regular solutions for a parabolic problem.

\subsection{Estimate I}\label{1-st estimate}
Multiply \eqref{2.1} by $u$ and integrate over $\mathcal{D}$ and $(0,t)$ to obtain
\begin{align}\label{E1}
\|u\|^2(t)+\int_0^t\int_{\mathcal{S}} u_{x}^2(0,y,z,\tau)\,dy\,dz d\tau=\|u_0\|^2,\ \ t\in(0,T).
\end{align}
 The following inequalities are crucial for our proof.

\begin{prop} \label{steklov} Let $v \in H^1_0(D).$ Then
\begin{align}
&\|v_y\|^2(t)\geq \frac {\pi^2}{B_y^2}\|v\|^2(t),\quad \|v_z\|^2(t)\geq \frac {\pi^2}{B_z^2}\|v\|^2(t),\nonumber\\ &\|v_x\|^2(t)\geq \frac {\pi^2}{L^2}\|v\|^2(t). \label{Estek}  
\end{align}
\end{prop}
\begin{proof} The proof is based on the Steklov inequality: let $v(t)\in H^1_0(0,\pi)$,  then $\int_0^{\pi}v_t^2(t)\,dt\geq\int_0^{\pi}v^2(t)\,dt.$
 Inequalities \eqref{Estek} follow from here by a simple scaling.
 \end{proof}

\subsection{Estimate II}\label{2-nd estimate}
Write the inner product
$$2\left(A u_,(1+x)u\right)(t)=0
$$
as
\begin{align}
&\frac{d}{dt}\left((1+x),u^2\right)(t)
+\int_{\mathcal{S}}u_x^2(0,y,z,t)\,dydz-c_s\|u\|^2(t)\nonumber\\
&+3\|u_x\|^2(t)+\|u_y\|^2(t)+\|u_z\|^2(t)
=\frac23(1,u^3)(t).\label{e2}
\end{align}
Making use of \eqref{lady1}, we compute
\begin{align}
I=&\frac{2}{3}(1,u^3)(t)
\le \frac{2}{3}\|u\|^3_{L^3(\mathcal{D})}(t)\le\frac{2^4}{3}\left[\|\nabla u\|^{1/2}(t)\|u\|^{1/2}(t)\right]^3\nonumber\\
&\le\frac{1}{8}\|\nabla u\|^2(t)+\frac{2^{17}}{3}\|u\|^6(t). 
\end{align}

Substituting $I$ into \eqref{e2}, we obtain

\begin{align*}
&\frac{d}{dt}\left((1+x),u^2\right)(t)
+\int_{\mathcal{S}}u_x^2(0,y,z,t)\,dydz\\
&+\frac{23}{8}\|u_x\|^2(t)+\frac78\|u_y\|^2(t)+\frac78\|u_z\|^2(t)\\
&-K_1\|u\|^6(t)-c_s\|u\|^2(t)\leq0,
\end{align*}
where $K_1=\frac{2^{17}}{3}.$

Using \eqref{Estek}, we get
\begin{align}
&\frac{d}{dt}\left((1+x),u^2\right)(t)+\frac{K_2}{2}\|u\|^2(t)
+\int_{\mathcal{S}}u_x^2(0,y,z,t)\,dydz\nonumber\\
&\big[\frac{K_2}{2}-c_s-K_1\|u\|^4(t)\big]\|u\|^2(t),
\end{align}
where
$$
K_1=\frac{2^{17}}{3}.$$

By conditions of Theorem \ref{theorem1}, the last inequality can be rewritten as

\begin{align*}
\frac{d}{dt}\left((1+x),u^2\right)(t)+2\chi((1+x),u^2)(t)\leq0,
\end{align*}
with $\chi=\frac{K_2}{4(1+L)}.$
Solving this inequality, we find
\begin{equation}
\|u\|^2(t)\leq((1+x),u^2)(t)\leq e^{-2\chi t}((1+x),u_0^2)\quad\forall t>0. \label{E2}
\end{equation}

\subsection{Estimate III}\label{3-d estimate}

Rewrite the scalar product

$$2\left(A u_,(1+x)u\right)(t)=0
$$
as

\begin{align*}
&\int_{\mathcal{S}}u_x^2(0,y,z,t)\,dydz
+3\|u_x\|^2(t)+\|u_y\|^2(t)+\|u_z\|^2(t)\nonumber\\
&=\frac23(1,u^3)(t)+c_s\|u\|^2(t)-2((1+x)u,u_t)(t).
\end{align*}

Using \eqref{lady1}, we find
$$I=\frac23(1,u^3)(t)\leq \frac12\|\nabla u\|^2(t)+\frac{2^{11}}{3}\|u\|^6¨(t).$$
Substituting $I$ into the last equation, we get
\begin{align}
&\int_{\mathcal{S}}u_x^2(0,y,z,t)\,dydz
+\frac52\|u_x\|^2(t)+\frac12\|u_y\|^2(t)+\frac12\|u_z\|^2(t)\nonumber\\
&\leq c_s\|u\|^2(t)+\frac{2^{11}}{3}\|u\|^6(t)-2((1+x)u,u_t)(t)\label{Enabla}
\end{align}
and
\begin{align}
\|u_x\|^2(t)&\leq \frac25\|u\|^2(t)\big[1+c_s+\frac{2^{11}}{3}\|u\|^4(t)\big]+\frac25(1+L)((1+x),u_t^2)(t)\nonumber\\&\le C_1\|u\|^2(t)+\frac25(1+L)((1+x),u_t^2)(t)\big], \label{Eux}
\end{align}
where
$$C_1=\frac25(1+c_s+\frac{2^{11}}{3}\|u_0\|^4).$$

\subsection{Estimate IV}\label{4-th estimate}

Write the inner prouct
$$((Au)_t,(1+x)u_t)(t)=0$$

as

\begin{align}
\frac{d}{dt}
&\left((1+x),u_t^2\right)(t)+\int_{\mathcal{S}}u_{xt}^2(0,y,z,t)\,dydz-c_s\|u_t\|^2(t)+3\|u_{xt}\|^2(t)\notag\\
&+\|u_{yt}\|^2(t)+\|u_{zt}\|^2(t)
+2\left((1+x)(uu_x)_t,u_t\right)(t)=0. \label {e4}
\end{align}

We calculate
\begin{align*}
&I=2((1+x)(uu_x)_t,u_t)(t)=2((1+x)(uu_t)_x,u_t)(t)\\&=((1+x)u_x-u,u_t^2)(t)\leq \|(1+x)u_x-u\|(t)\|u_t\|^2_{L^4(\mathcal{D})}(t)\\&
\leq (1+L)\big[\|u_x\|(t)+\|u\|(t)\big]4^{3/2}\|u_t\|^{1/2}(t)\|\nabla u_t\|^{3/2}(t)\\&
\leq\frac34\epsilon^{4/3}\|\nabla u_t\|^2(t)+\frac{1}{4\epsilon^4}4^6(1+L)^4\big[\|u_x\|(t)+\|u\|(t)\big]^4\|u_t\|^2(t).
\end{align*}

Taking $\frac34\epsilon^{4/3}=\frac18,$  we get
\begin{align*}
&I\leq \frac18\|\nabla u_t\|^2(t)+(1+L)^43^32^{16}\big[\|u_x\|^4(t)+\|u\|^4(t)\big]\|u_t\|^2(t).
\end{align*}

Substituting $I$ into \eqref{e4} and making use of \eqref{Eux}, we obtain

\begin{align*}
&\frac{d}{dt}\left((1+x),u_t^2\right)(t)+\int_{\mathcal{S}}u_{xt}^2(0,y,z,t)\,dydz+\frac14\big\{\frac{23}{8}\|u_{xt}\|^2(t)\\&+\frac78\|u_{yt}\|^2(t)+\frac78\|u_{zt}\|^2(t)\big\}+\frac34K_2-\big\{c_s+(1+L)^43^32^{16}(2C_1^2+1)\|u_0\|^4\\&+\frac{3^32^{19}}{25}(1+L)^6((1+x),u_t^2)^2(t)\big\}\|u_t\|^2(t)\leq 0
\end{align*}
which can be rewritten as
¨\begin{align}
&\frac{d}{dt}\left((1+x),u_t^2\right)(t)+\int_{\mathcal{S}}u_{xt}^2(0,y,z,t)\,dydz+\frac14\big\{\frac{23}{8}\|u_{xt}\|^2(t)\notag\\&
+\frac78\|u_{yt}\|^2(t)+\frac78\|u_{zt}\|^2(t)\big\}+\big[\frac{3K_2}{4}-c_s-K_3\|u_0\|^4\notag\\&-K_4((1+x),u_t^2)^2(t)\big]\|u_t\|^2(t)\leq 0,
\label{eut}
\end{align}
where
\begin{align}
K_3=3^32^{16}(1+L)^4(2C_1^2+1),\quad K_4=\frac{3^32^{19}(1+L)^6}{25}.\label{EK}
\end{align}
Due to conditions of Theorem \ref{theorem1},
$$K_3\|u_0\|^4< \frac{K_2}{4}, \quad K_4((1+x),u_t^2)^2(0)<\frac{K_2}{4},$$
hence, \cite{familark}, 
$$ K_4((1+x),u_t^2)^2(t)<\frac{K_2}{4}\quad \forall t> 0$$
and \eqref{eut} becomes
\begin{align*}
&\frac{d}{dt}\left((1+x),u_t^2\right)(t)+\frac14\big\{\frac{23}{8}\|u_{xt}\|^2(t)+\frac78\|u_{yt}\|^2(t)\notag\\&+\frac78\|u_{zt}\|^2(t)\big\}\leq0.
\end{align*}
Making use of \eqref{Estek}, we get
 
$$\frac{d}{dt}\left((1+x),u_t^2\right)(t)+\frac{K_2}{4(1+L)}((1+x),u_t^2)(t)\leq 0.$$
Since
$$((1+x),u_t^2)(0)=((1+x),\big[(c_s +u_0)u_{0x}+\Delta u_{0x}\big]^2)\leq J_0,$$
solving this inequality, we find
\begin{align}
\|u_t\|^2\leq ((1+x),u_t^2)(t)\leq e^{-\chi t}((1+x),u_t^2)(0)\leq e^{-\chi t}J_0\label{EUt}
\end{align}

with $\chi=\frac{K_2}{4(1+L)}.$

Returning to \eqref{eut} and taking into account \eqref{E2}, \eqref{Enabla}, we obtain

\begin{align}
&((1+x),u_t^2)(t)+\int_{\mathcal{S}}u_x^2(0,y,z,t)\,dydz
+\| u\|_{H^1_0(\mathcal{D})}^2(t)\notag\\
&+\int_{\mathcal{S}}u^2_x(0,y,z,t)\,dydz\leq C_2e^{-\chi t}J_0;\label{EH1}\\
 &\int_0^t\big\{\int_{\mathcal{S}}u^2_x(0,y,z,\tau)\,dydz+\|\nabla u_{\tau}\|^2(\tau)\big\}\,d\tau\notag\\&\leq C_3J_0\quad\forall \,t>0,\label{EUT}
\end{align}
where the constants $C_2,C_3$ do not depend on $t>0$.

\subsection{Estimate V}\label{5-d estimate}

Transform the scalar product 

$$-2((1+x)Au,u_{yy}+u_{zz})(t)=0$$
 into the following equality:
 \begin{align}
 &-c_s\big(\|u_y\|^2(t)+\|u_z\|^2(t)\big)+\int_{\mathcal{S}}\big[u^2_{xy}(0,y,z,t)+u^2_{xz}(0,y,z,t)\big]\,dydz\notag\\
 &+\|u_{yy}\|^2(t)+\|u_{zz}\|^2(t)+2\|u_{yz}\|^2(t)+3\|u_{xy}\|^2(t)+3\|u_{xz}\|^2(t)\notag\\&
 +((1+x)u_x-u,u^2_y)(t)+((1+x)u_x-u,u^2_z)(t)\notag\\&
 =2((1+x)u_t,u_{yy}+u_{zz})(t). \label{e5}
\end{align}

We estimate
\begin{align*}
&I_1=((1+x)u_x-u,u^2_y)(t)\leq \|(+x)u_x-u\|(t)\|u_y\|^2_{L^4(\mathcal{D})}(t)\\&
\leq (1+L)\big[\|u_x\|(t)+\|u\|(t)\big]4^{3/2}C^2_D\|\nabla u\|^{1/2}(t)\|\nabla 
u_y\|^{3/2}(t)\\&
\leq \frac18\|\nabla u_y|^2(t)+(1+L)^4C^8_D2^{16¨}3^3\big[\|u_x\|^4(t)+\|u\|^4(t)\big]\|\nabla u\|^2(t).
\end{align*}            
Similarly,
\begin{align*}
&I_2=((1+x)u_x-u,u^2_z)(t)\leq  \frac18\|\nabla u_z|^2(t)\\&+(1+L)^4C^8_D2^{16¨}3^3\big[\|u_x\|^4(t)+\|u\|^4(t)\big]\|\nabla u\|^2(t).
\end{align*}

Substituting $I_1,I_2$ into \eqref{e5},we find
\begin{align*}
&\int_{\mathcal{S}}\big[u^2_{xy}(0,y,z,t)+u^2_{xz}(0,y,z,t)\big]\,dydz
+\|u_{yy}\|^2(t)+\|u_{zz}\|^2(t)\\&+\|u_{yz}\|^2(t)+\|u_{xy}\|^2(t)+\|u_{xz}\|^2(t)\\&\leq C_4(L)\big[\|\nabla u\|^6(t)+\|u\|^4(t)\|\nabla u\|^2(t)+((1+x),u_t^2)(t)].
\end{align*}

Making use of \eqref{E2}, \eqref{Enabla}, \eqref{EUt}, \eqref{EH1},  we get
\begin{align}
&\int_{\mathcal{S}}\big[u^2_{xy}(0,y,z,t)+u^2_{xz}(0,y,z,t)\big]\,dydz + \|u_y\|^2_{H^1(\mathcal{D})}(t)\notag\\&+\|u_z\|^2_{H^1(\mathcal{D})}(t)\leq 
C_5(L,J_0)e^{- \chi t}J_0.\label{EH2´} 
\end{align}

To prove that
$$\|u\|^2_{H^2(\mathcal{D})}(t)\leq C_6(L,J_0)e^{- \chi t}J_0,$$
it is necessary to estimate $\|u_{xx}\|(t).$

\subsection{Estimate VI}\label{6-d estimate}

¨Strong solutions to \eqref{2.1}-\eqref{2.4} satisfy the following elliptic problem:
\begin{align}
& \Delta u_x=-u_t-c_su_x -\frac12(u^2)_x;\label{deltau}\\
& u_x(0,y,z,t)=\phi(y,z,t);\label{phi}\\
&u_x(L,y,z,t)=u_x(x,0,z,t)=u_x(x,B_y,z,t)\\&=u_x(x,y,0,t)=u_x(x,y,B_z,t)=0.
\end{align}

Denote $v=u_x-\phi(y,z)\big(1-\frac xL\big)$ to come to the following Dirichlet problem:

\begin{align}
& \Delta v=-u_t-c_su_x -\frac12(u^2)_x-\big(1-\frac xL\big)(\phi_y(y,z,t))_y\notag\\&-\big(1-\frac xL\big)(\phi_z(y,z,t))_z\equiv F(x,y,z,t);\label{deltav}\\
& v|_{\gamma}=0.
\end{align}

Considering the scalar product
$$-(\Delta v,v)(t)=-(F,v)(t),$$

we find

\begin{align}
&\|v_x\|^2(t)+\|v_y\|^2(t)+\|v_z\|^2(t)\leq C_7\big\{ \|u_t\|^2(t)+\|u\|^2(t)
+\|u\|^4_{L^4{(\mathcal{D})}}(t)\notag\\&+\int_{\mathcal{S}}[(u^2_{xy}(0,y,z,t)
+u^2_{xz}(0,y,z,t)]\,dydz\big\}.\label{nablav}
\end{align}

We estimate
$$\|u\|^4_{L^4{(\mathcal{D})}}(t)\leq 4^3\|u\|(t)\|\nabla u\|^3(t).$$

Substituting this into \eqref{nablav} and making use of \eqref{EUt},\eqref{EH1}, we get
$$\|v_x\|^2(t)\leq C_8e^{-\chi t}J_0.$$

By definition,
$$ u_{xx}=v_x-\frac 1L\phi(y,z,t).$$
Hence, due to \eqref{EH2´},
$$\|u_{xx}\|^2(t)\leq C_9e^{-\chi t}J_0$$
 which jointly with \eqref{EH2´} reads
 \begin{align}
 \|u\|^2_{H^2(\mathcal{D}}(t)\leq C_{10}(L,J_0)e^{-\chi t} .\label{EH2}
 \end{align}

\subsection{Estimate VII}\label{7-d estimate}
Consider the scalar product
$$2((1+x)Au,\partial^4_y u+\partial^4_z u+\partial^2_y\partial^2_z u )(t)=0$$
and transform it into the equality
\begin{align}
&\frac{d}{dt}((1+x),u^2_{yy}+u^2_{zz}+u^2_{yz})(t)\notag\\&
-c_s\big[\|u_{yy}\|^2(t)+\|u_{zz}\|^2(t)+]\|u_{yz}\|^2(t)\big]\notag\\&
+\int_{\mathcal{S}}\big\{u^2_{xyy}(0,y,z,t)+u^2_{xzz}(0,y,z,t)+u^2_{xyz}(0,y,z,t)
\big\}\,dydz\notag\\&
+\|u_{yyy}\|^2(t)+\|u_{zzz}\|^2(t)+2\|u_{yzz}\|^2(t)+2\|u_{zyy}\|^2(t)\notag\\&
+3\|u_{xyy}\|^2(t)+3\|u_{xzz}\|^2(t)+3\|u_{xyz}\|^2(t)\notag\\&
+2((1+x)uu_x,\partial^4_y u)(t)+2((1+x)uu_x,\partial^4_z u)(t)\notag\\&+2((1+x)uu_x,\partial^2_y\partial^2_z u)(t)=0.\label{e7}
\end{align}
We estimate
\begin{align*}
&I_1=2((1+x)uu_x,\partial^4_z u)(t)=-2((1+x)u_z u_x,\partial^3_z u)(t)\notag\\&
-2((1+x)uu_{zx},\partial^3_z u)(t)=((1+x)u_x-u,u^2_{zz})(t)\notag\\&
-2(u^2_z,u_{zz})(t)-2((1+x)u^2_z,u_{xzz})(t)\equiv I_{11}+I_{12}+I_{13},\\
\end{align*}
where
\begin{align*}
&I_{11}=((1+x)u_x-u,u^2_{zz})(t)\leq \|(1+x)u_x-u\|(t)\|u_{zz}\|^2_{L^4(\mathcal{D})}(t)\\&\leq (1+L)[\|u_x\|(t)+\|u\|(t)]4^{3/2}\|u_{zz}\|^{1/2}(t)\|\nabla u_{zz}\|^{3/2}(t)\\&
\leq\frac34\epsilon^{4/3}\|\nabla u_{zz}\|^2(t)+\frac{(1+L)^4}{\epsilon^4}2^{13}\|u_{zz}\|^2(t)\big[\|u_x\|^4+\|u\|^4\big];
\end{align*}

\begin{align*}
&I_{12}=2(u^2_z,u_{zz})(t)\leq 2^4\|u_{zz}\|(t)\|u_z\|^{1/2}(t)\|\nabla  u_z\|^{3/2}(t)\leq C\|u\|^3_{H^2(\mathcal{D})}(t)
\end{align*}
and
\begin{align*}
&I_{13}=2((1+x)u^2_z,u_{xzz})(t)\leq 4^{3}(1+L)\|u_{xzz}\|(t)\|u_z\|^{1/2}(t)\|\nabla  u_z\|^{3/2}(t)\\&\leq \delta_1 \|u_{xzz}\|^2(t)+\frac{4^3(1+L)^2}{\delta_1}\|u_z\|(t)\|\nabla u_z\|^3(t).
\end{align*}

Similarly,

\begin{align*}
&I_2=2((1+x)uu_x,\partial^4_y u)(t)=-2((1+x)u_y u_x,\partial^3_y u)(t)\notag\\&
-2((1+x)uu_{yx},\partial^3_y u)(t)=((1+x)u_x-u,u^2_{yy})(t)\notag\\&
-2(u^2_y,u_{yy})(t)-2((1+x)u^2_y,u_{xyy })(t)\equiv I_{21}+I_{22}+I_{23},
\end{align*}

where

\begin{align*}
&I_{21}=((1+x)u_x-u,u^2_{yy})(t)\leq \|(1+x)u_x-u\|(t)\|u_{yy}\|^2_{L^4(\mathcal{D})}(t)\\&
\leq\frac34\epsilon^{4/3}\|\nabla u_{yy}\|^2(t)+\frac{(1+L)^4}{\epsilon^4}2^{13}\|u_{yy}\|^2(t)\big[\|u_x\|^4(t)+\|u\|^4(t)\big];
\end{align*}

\begin{align*}
&I_{22}=2(u^2_y,u_{yy})(t)\leq 2^4\|u_{yy}\|(t)\|u_y\|^{1/2}(t)\|\nabla  u_y\|^{3/2}(t)\leq 2^4\|u\|^3_{H^2(\mathcal{D})}(t)
\end{align*}
and
\begin{align*}
&I_{23}=2((1+x)u^2_y,u_{xyy})(t)\leq \delta_1 \|u_{xyy}\|^2(t)+\frac{4^3(1+L)^2}{\delta_1}\|u_y\|(t)\|\nabla u_y\|^3(t).
\end{align*}
\begin{align*}
I_3=&2((1+x)uu_x,\partial^2_y\partial^2_z u)(t)=-2((1+x)u_z u_x,\partial^2_y\partial_z u)(t)\\&
-2((1+x)uu_{zx},\partial^2_y\partial_z u)(t)=2((1+x)u_x,u^2_{zy})(t)\\&
-2((1+x)u^2_z,u_{xyy})(t)-(u^2_z,u_{yy})(t)+2((1+x)uu_{xzz},u_{yy})(t)\\&\equiv I_{31}+I_{32}+I_{33}+I_{34}.
\end{align*}

From here

\begin{align*}
I_{31}=&2((1+x)u_x,u^2_{zy})(t)\leq 2(1+L)\|u_x\|(t)\|u_{zy}\|^2_{L^4(\mathcal{D})}(t)\\&
\leq 2^4(1+L)\|u_x\|(t)\|u_{zy}\|^{1/2}(t)\|\nabla u_{zy}\|^{3/2}(t)\\&
\leq \frac34 \epsilon^{4/3}\|\nabla u_{zy}\|^2(t)+\frac{2^{14}}{\epsilon^4}(1+L)^4\|u_x\|^4(t)\|u_{zy}\|^2(t);
\end{align*}

\begin{align*}
&I_{32}=-2((1+x)u^2_z,u_{xyy})(t)\leq (1+L)\|u_{xyy}\|(t)\|u_z\|^2_{L^4(\mathcal{D})}(t)\\&
\leq \delta\|u_{xyy}\|^2(t)+\frac{4^3 (1+L)^2}{\delta}\|\nabla u\|(t)\|\nabla u_z\|^3(t);
\end{align*}

\begin{align*}
&I_{33}=(u^2_z,u_{yy})(t)\leq \|u_{yy}\|(t)\|u_z\|^2_{L^4(\mathcal{D})}(t)\\&
\leq \|u_{yy}\|^2(t)+4^2\|u_z\|(t)\|\nabla u_z\|^3(t);
\end{align*}

\begin{align*}
&I_{34}=2((1+x)uu_{xzz},u_{yy})(t)\leq 2(1+L)\|u_{xzz}\|(t)\|u\|_{L^4(\mathcal{D})}(t)\|u_{yy}\|_{L^4(\mathcal{D})}(t)\\&
\leq \delta_1\|u_{xzz}\|^2(t)+\frac{4^3(1+L)^2}{\delta_1}\|\nabla u_{yy}\|^{3/2}(t)\|u\|^{1/2}(t)\|\nabla u\|^{3/2}(t)\|u_{yy}\|^{1/2}(t)\\&
\leq \delta_1\|u_{xzz}\|^2(t)+\frac{3\delta^{4/3}}{4\delta_1}\|\nabla u_{yy}\|^2(t)
+\frac{4^{11}(1+L)^8}{\delta_1 \delta^4}\|u\|^2(t)\|u_{yy}\|^2(t)\|\nabla  u\|^6(t),
\end{align*}
 where $\delta, \delta_1, \epsilon$ are arbitrary positive numbers.

Taking them sufficiently small, we reduce \eqref{e7} to the form

\begin{align*}
&\frac{d}{dt}((1+x),u^2_{yy}+u^2_{zz}+u^2_{yz})(t)\\&
+\int_{\mathcal{S}}\big\{u^2_{xyy}(0,y,z,t)+u^2_{xzz}(0,y,z,t)+u^2_{xyz}(0,y,z,t)
\big\}\,dydz\\&
+\|u_{yyy}\|^2(t)+\|u_{zzz}\|^2(t)+\|u_{yzz}\|^2(t)+\|u_{zyy}\|^2(t)\\&
+\|u_{xyy}\|^2(t)+\|u_{xzz}\|^2(t)+\|u_{xyz}\|^2(t)\\&
\leq C_{11}(L)\|u\|^3_{H^2(\mathcal{D})}(t)\big[1+\|u\|^3_{H^2(\mathcal{D})}(t)].
\end{align*}

Integrating this, we obtain

\begin{align}
&((1+x),u^2_{yy}+u^2_{zz}+u^2_{yz})(t)\notag\\&
+\int_0^t\big\{\int_{\mathcal{S}}\big\{u^2_{xyy}(0,y,z,\tau)+u^2_{xzz}(0,y,z,\tau)+u^2_{xyz}(0,y,z,\tau)
\big\}\,dydz\notag\\&
+\|u_{yyy}\|^2(\tau)+\|u_{zzz}\|^2(\tau)+\|u_{yzz}\|^2(\tau)+\|u_{zyy}\|^2(\tau)\notag\\&
+\|u_{xyy}\|^2(\tau)+\|u_{xzz}\|^2(\tau)+\|u_{xyz}\|^2(\tau)\big\}\,d\tau\notag\\&
\leq C_{12}(L,B_y,B_z,J_0)((1+x),u^2_{0yy}+u^2_{0zz}+u^2_{0zy}),\label{E7}\end{align}

with the constant $C_{12}(L,B_y,B_z,J_0)$ independent of $t>0.$

\begin{lemma}\label{last}
 Strong solutions to \eqref{2.1}-\eqref{2.4} satisfy the following inclusions:
\begin{align}
&u\in L^{\infty}(0,\infty;H^2(\mathcal{D}))\cap L^2(0,\infty;H^3(\mathcal{D}));\label{Estrong}\\&
\Delta u_x \in  L^{\infty}(0,\infty;L^2(\mathcal{D}))\cap
L^2(0,\infty;H^1(\mathcal{D}));\label{EDelta}.
\end{align}
\end{lemma}
\begin{proof} First, we will prove \eqref{Estrong}. For this purpose, rewrite \eqref{2.1} in the form
\begin{align}
&\Delta v=u_t-c_su_x-uu_x-(1-\frac xL)\Phi_{yy}(y,z,t)\notag\\&-(1-\frac xL)\Phi_{zz}(y,z,t)\equiv F(x,y,z,t);\label{Deltav}\\&
v|_{\gamma}=0,\label{Dirichlet}
\end{align}

where 
$$ \Phi(y,z,t)\equiv u_x(0,y,z,t),\quad v=u_x-\Phi(y,z,t)(1-\frac xL).$$
Due to \eqref{E7}, $$F\in L^2(0,\infty;L^2(\mathcal{D})).$$
This implies that the Dirichlet problem \eqref{Deltav},\,\eqref{Dirichlet} has a unique solution
$$v\in L^2(0,\infty;H^2(\mathcal{D})),$$
hence,
$$ u_x \in L^2(0,\infty;H^2(\mathcal{D})).$$
Taking into acount \eqref{E7}, we prove \eqref{Estrong}. On the other hand, it follows from the equation
$$ \Delta u_x=-u_t-c_su_x-uu_x \equiv G(x,y,t)$$ that $G\in L^2(0,\infty;H^1(\mathcal{D})).$\par Indeed, making use of Proposition \ref{lady1} and \eqref{Estrong}, we find
\begin{align*}
&\|uu_x\|_{H^1(\mathcal{D})}(t)\leq \|uu_x\|(t)+\|\nabla (uu_x)\|(t)\\&
\leq \|u\|_{L^4(\mathcal{D})}(t)\|u_x\|_{L^4(\mathcal{D})}(t)+\|u\|_{L^4(\mathcal{D})}(t)
\|\nabla u_x\|_{L^4(\mathcal{D})}(t)+\|\nabla u\|^2_{L^4(\mathcal{D})}(t)\\&\leq 2^3\big[\|u\|^{1/4}(t)\|\nabla u\|(t)\|\nabla u_x\|^{3/4}(t)\\&
+\|u\|^{1/4}(t)\|\nabla u\|^{3/4}(t)\|\nabla u_x\|^{1/4}(t)\|\nabla^2 u_x\|^{3/4}(t)+\|\nabla u\|^{1/2}(t)\|\nabla^2 u\|^{3/2}(t)\big]\\&\leq 2^3 \|u\|^2_{H^3(\mathcal{D})}(t)\leq \infty.
\end{align*}
By \eqref{Estrong}, $uu_x \in L^2(0,\infty;H^1(\mathcal{D})).$
This implies that $G\in L^2(0,\infty;H^1(\mathcal{D})).$
Thus, the proof of Lemma \ref{last} is complete.
\end{proof}

The proof of the existence pat of Theorem \ref{theorem1} is also complete.

\begin{lemma} \label{uniq}
The strong solution of Theorem \ref{theorem1} is uniquelly defined.
\end{lemma}
\begin{proof} Let $u_1,\; u_2$  be distinct solutions to \eqref{2.1}-\eqref{2.4}, then $w=u_1-u_2$ satisfies the following initial-boundary value problem:
\begin{align}
&  Lw\equiv w_t+c_sw_x+\Delta w_x+ww_x+(u_2 w)_x=0;\;\mbox{in}\;\mathcal{Q}_t; \label{Lv}\\
&w|_{\gamma}=w_x(L,y,z,t)=0,\quad t>0,\: (y,z)\in \mathcal{S};\\
&w(x,y,x,0)\equiv 0 \quad (x,y,z)\in \mathcal{D}. \label{indata}
\end{align}

Consider the scalar product
$$2((1+x)Lw,w)(t)=0$$ which can be transformed into the following equality:
\begin{align}
&\frac{d}{dt}((1+x),w^2)(t)+\int_{\mathcal{S}}w^2_x(0,y,z,t)\,dydz
+2\|w_x\|^2(t)+\|\nabla w\|^2(t)\notag\\&-\frac23(1,w^3)(t)
+((1+x)u_{2x}-u_2,w^2)(t)=0. \label{euniq}
\end{align}

Making use of Proposition \ref{lady1} and acting in the same manner as above, we find
\begin{align*}
&I_1=-\frac23(1,w^3)(t)\leq \frac34 \epsilon^{4/3}\|\nabla w\|^2(t)+\frac{2^{14}}{3^4\epsilon^4}\|w\|^6(t)
\end{align*}

and
\begin{align*}
&I_2=2((1+x)u_{2x}-u_2,w^2)(t)\leq \frac34 \epsilon^{4/3}\|\nabla w\|^2(t)\\&
+\frac{C(L)}{\epsilon^4}\big[ \|u_{2x}\|^4(t)+\|u_2\|^4(t)\big]\|w\|^2(t).
\end{align*}

Taking $\epsilon>0$ sufficiently small, substituting $I_1, I_2$ into \eqref{euniq} and making use of \eqref{Estrong}, we obtain
\begin{align*}
&\frac{d}{dt}((1+x),w^2)(t)\leq C_{13}(L)\big[\|u_{2x}\|^4(t)\\&
+\|u_1\|^4(t)+\|u_2\|^4(t)\big]((1+x),w^2)(t).
\end{align*}
By \eqref{EH2´},
$$\|u_{2x}\|^4(t)+\|u_1\|^4(t)+\|u_2\|^4(t) \in L^1(0,\infty)$$ and
by the Gronwall Lemma,
$$\|w\|^2(t)\leq ((1+x),w^2)(t)\equiv 0.$$
The proof of Lemma \ref{uniq} is complete.
\end{proof}
Obviously, this completes also the proof of Theorem \ref{theorem1}. 
\end{proof}

\begin{rem}
If $w(x,y,z,0)=w _0(x,y,z)\ne 0,$ then
$$\|w\|^2(t)\leq ((1+x),w^2)(t)\leq C(L,J_0)((1+x),w^2_0)\;\forall t>0.$$
This means continuous dependence of solutions to \eqref{2.1}-\eqref{2.4} on initial data.
\end{rem}

\begin{rem}

To prove Theorem 1, we have the following alternative: either the coefficient $K_2$ to be sufficiently large which can be done for small values of $L,B_y,B_z$ (at least one of them) or initial data $ \|u_0\|,J_0$ to be sufficiently small while $c_s$ is fixed. In our case, $L$ may be arbitary large finite, hence we can handle values of $B_y,B_z.$
\end{rem}

\


\medskip



\begin{thebibliography}{99}

\bibitem{bona2}
\newblock J. L. Bona and R. W. Smith,
\newblock The initial-value problem for the Korteweg-de Vries equation,
\newblock Phil. Trans. Royal Soc. London Series A 278 (1975), 555--601.

\bibitem{bona1}
\newblock J. L. Bona, S. M. Sun and B.-Y. Zhang,
\newblock A nonhomogeneous boundary-value problem for the Korteweg-de Vries equation posed on a finite domain,
\newblock Comm. Partial Differential Equations 28 (2003), 1391--1436.

\bibitem{bona3}
\newblock J. L. Bona, S. M. Sun and B.-Y. Zhang,
\newblock Nonhomogeneous problems for the Korteweg-de Vries and the Korteweg-de Vries-Burgers equations in a quarter plane,
\newblock Ann. Inst. H. Poincar\'{e} Anal. Non Lin\'{e}aire 25 (2008), 1145--1185.

\bibitem{bourgain2}
\newblock J. Bourgain,
\newblock On the compactness of the support of solutions of dispersive equations,
\newblock Int. Math. Res. Notices 9 (1997), 437--447.

\bibitem{bubnov}
\newblock B. A. Bubnov,
\newblock Solvability in the large of nonlinear boundary-value problems for
 the Kortewegde Vries equation in a bounded domain (Russian),
\newblock Differentsial`nye uravneniya 16 (1980), 34--41.
\newblock Engl. transl. in: Diff. Equations 16 (1980), 24--30.


\bibitem{colin}
\newblock T. Colin and J.-M. Ghidaglia,
\newblock An initial-boundary-value problem for the Korteweg-de Vries Equation posed on a finite interval,
\newblock Adv. Differential Equations 6 (2001), 1463--1492.





\bibitem{doronin1}
\newblock G. G. Doronin and N. A. Larkin,
\newblock KdV equation in domains with moving boundaries,
\newblock J. Math. Anal. Appl. 328 (2007), 503--515.


\bibitem{faminski}
\newblock A. V. Faminskii,
\newblock The Cauchy problem for the Zakharov-Kuznetsov equation (Russian),
\newblock Differentsialnye Uravneniya, 31 (1995), 1070--1081;
\newblock Engl. transl. in: Differential Equations 31 (1995), 1002--1012.

\bibitem{faminski2}
\newblock A. V. Faminskii,
\newblock Well-posed initial-boundary value problems for the Zakharov-Kuznetsov equation,
\newblock Electronic Journal of Differential equations 127 (2008), 1--23.

\bibitem{familark}
\newblock A. V. Faminskii and N. A. Larkin,
\newblock Initial-boundary value problems for quasilinear dispersive equations posed on a bounded interval,
\newblock Elec. J. Diff. Equations 2010 (2010), 1--20.

\bibitem{farah}
\newblock L. G. Farah, F. Linares and A. Pastor,
\newblock A note on the 2D generalized Zakharov-Kuznetsov equation: Local, global, and scattering results,
\newblock J. Differential Equations 253 (2012), 2558--2571.



\bibitem{kato}
\newblock T. Kato,
\newblock On the Cauchy problem for the (generalized) Korteweg-de- Vries equations,
\newblock Advances in Mathematics Suplementary Studies, Stud. Appl. Math. 8 (1983), 93--128.

\bibitem{ponce2}
\newblock C. E. Kenig, G. Ponce and L. Vega,
\newblock Well-posedness and scattering results for the generalized Korteweg-de Vries equation
 and the contraction principle,
\newblock Commun. Pure Appl. Math. 46 (1993), 527--620.



\bibitem{lady}
\newblock O. A. Ladyzhenskaya,
\newblock The Boundary Value Problems of Mathematical Physics.
\newblock Applied Math. Sci. 49, Springer-Verlag, New York, 1985.

\bibitem{lady2}
\newblock O. A. Ladyzhenskaya, V. A. Solonnikov and N. N. Uraltseva,
\newblock Linear and Quasilinear Equations of Parabolic Type.
\newblock American Mathematical Society, Providence, Rhode Island, 1968.



\bibitem{larh1}
\newblock N. A. Larkin,
\newblock Exponential decay of the $H^1$-norm for the 2D Zakharov-Kuznetsov equation on a half-strip,
\newblock J. Math. Anal. Appl. 405 (202013), 326--¨335.

\bibitem{larkin}
\newblock N. A. Larkin,
\newblock Korteweg-de Vries and Kuramoto-Sivashinsky Equations in Bounded Domains,
\newblock J. Math. Anal. Appl. 297 (2004), 169--185.

\bibitem{lar2}
\newblock N. A. Larkin, E. Tronco,
\newblock Nonlinear quarter-plane problem for the Korteweg-de Vries equation,
\newblock Electron. J. Differential Equations 2011 (2011), 1--22.

\bibitem{larkintronco}
\newblock N. A. Larkin and E. Tronco,
\newblock Regular solutions of the 2D Zakharov-Kuznetsov equation on a half-strip,
\newblock J. Differential Equations 254 (2013), 81--101.

\bibitem{pastor}
\newblock F. Linares and A. Pastor,
\newblock Local and global well-posedness for the 2D generalized Zakharov-Kuznetsov equation,
\newblock J. Funct. Anal. 260 (2011), 1060--1085.

\bibitem{pastor2}
\newblock F. Linares, A. Pastor and J.-C. Saut,
\newblock Well-posedness for the ZK equation in a cylinder and on the background of a KdV Soliton,
\newblock Comm. Part. Diff. Equations 35 (2010), 1674--1689.


\bibitem{saut}
\newblock F. Linares and J.-C. Saut,
\newblock The Cauchy problem for the 3D Zakharov-Kuznetsov equation,
\newblock Disc. Cont. Dynamical Systems A 24 (2009), 547--565.

 
\bibitem{ribaud}
\newblock F. Ribaud and S. Vento,
\newblock Well-posedness results for the three-dimensional Zakharov-Kuznetsov equation,
\newblock SIAM J. Math. Anal., 44 (2012),2289--2304.



\bibitem{rivas}
\newblock I. Rivas, M. Usman and B.-Y. Zhang,
\newblock Global
 well-posedness and asymptotic behavior of a class of
 initial-boundary value problem for the Korteweg-de Vries equation
 on a finite domain,
\newblock Math. Control Related Fields 1 (2011), 61--81.



\bibitem{rosier1}
\newblock L. Rosier,
\newblock A survey of controllability and stabilization results for partial differential equations,
\newblock RS - JESA 41 (2007), 365--411.

\bibitem{rozan}
\newblock L. Rosier and B.-Y. Zhang,
\newblock Control and stabilization of the KdV equation: recent progress,
\newblock J. Syst. Sci. Complexity 22 (2009), 647--682.

\bibitem{saut2}
\newblock J. C. Saut,
\newblock Sur quelques g\'{e}n\'{e}ralisations de l'\'{e}quation de Korteweg-de Vries (French),
\newblock J. Math. Pures Appl. 58 (1979), 21--61.


\bibitem{temam}
\newblock J.-C. Saut and R. Temam,
\newblock An initial boundary-value problem for the Zakharov-Kuznetsov equation,
\newblock Advances in Differential Equations 15 (2010), 1001--1031.

\bibitem{temam2}
\newblock J.-C. Saut, R. Temam and C. Wang,
\newblock An initial and boundary-value problem for the Zakharov-Kuznetsov equation in a bounded domain,
\newblock J. Math. Phys. 53 115612(2012).

\bibitem{temam1}
\newblock R. Temam,
\newblock Sur un probl\`{e}me non lin\'{e}aire (French),
\newblock J. Math. Pures Appl. 48 (1969), 159--172.

\bibitem{wang}
\newblock C. Wang,
\newblock Local existence of strong solutions to the 3D Zakharov-Kuznetsov equation in a bounded domains,
\newblock Appl. Math. Optim.,  69 (2014), 1--19.

\bibitem{zk}
\newblock V. E. Zakharov and E. A. Kuznetsov,
\newblock On three-dimensional solitons,
\newblock Sov. Phys. JETP 39 (1974), 285--286.



\bibitem{zhang}
\newblock B.-Y. Zhang,
\newblock Exact boundary controllability of the Korteweg-de Vries equation,
\newblock SIAM J. Control Optim. 37 (1999), 543--565.






\end{thebibliography}
  \end{document}